\documentclass[10pt,reqno]{amsart}
\usepackage{amssymb}
\renewcommand{\leq}{\leqslant}
\renewcommand{\geq}{\geqslant}
\makeatletter
\numberwithin{equation}{section}
\numberwithin{figure}{section} 
\theoremstyle{plain}
\newtheorem{thm}{Theorem}[section]
\newtheorem{cor}[thm]{Corollary}
\newtheorem{lem}[thm]{Lemma} 
\theoremstyle{definition}
\newtheorem{rem}{Remark}
\begin{document}
\title{On the spaces of $\lambda-$ convergent and bounded series}
\author{Meltem Kaya and Hasan Furkan$^*$}
\thanks{*Corresponding author}
\address[M. Kaya]{Kahramanmara\c s S\"ut\c c\"u \.Imam  \"Un\.ivers\.ites\.i, Fen B\.il\.imler\.i Enst\.it\"us\"u
, 46100--Kahramanmara\c s, T\"urkiye}
\email[M. Kaya]{meltemkaya55@hotmail.com}
\address[H. Furkan]{Kahramanmara\c s S\"ut\c c\"u \.Imam  \"Un\.ivers\.ites\.i, Fen Edebiyat Fak\"ultes\.i, 46100--Kahra-manmara\c s, T\"urkiye}
\email[H. Furkan]{hasanfurkan@ksu.edu.tr, hasanfurkan@hotmail.com}

\begin{abstract}
The main purpose of this study is to introduce the  spaces $cs^{\lambda},~cs_0^{\lambda}$ and $bs^{\lambda}$ which are $BK-$spaces of non-absolute type. We prove  that these spaces are linearly isomorphic to the spaces $cs, cs_0$ and $bs$, respectively and derive some inclusion relations. Additionally, their Schauder bases have been constructed and the $\alpha-,\beta-$  and $\gamma-$ duals  of these spaces  have been computed.  Finally, we characterize some matrix classes from the spaces $cs^{\lambda},~cs_0^{\lambda}$ and $bs^{\lambda}$ to spaces $\ell_p,~ c$ and $c_0$, where $1\leq p\leq\infty$ .
\end{abstract}
\maketitle
\section{Preliminaries, Background  and Notation}
By a \emph{sequence space}, we understand a linear subspace of the
space $w=\mathbb{C}^\mathbb{N}$, where
$\mathbb{N}=\{0,1,2,\dots\}$. A sequence spaces $E$ with a linear topology is called a $K-$ spaces provided
each of the maps $p_i: E\to\mathbb{C}$ defined $p_i(x)=x_i$ is continuous for all $i\in\mathbb{N}$, where $\mathbb{C}$ denotes the complex field
and $\mathbb{N}=\{0,1,2,\dots\}$. A $K-$ space is called an $FK-$ space provided $E$ is a complete linear metric space. An $FK-$ spaces whose topology
is normable is called a $BK-$ space (see \cite[pages 272-273]{kuldary}) which
contains $\phi$, the set of all finitely non--zero sequences.
 We write $\ell_\infty $, $c$ and $c_0$
for the spaces of all bounded, convergent and null sequences, respectively.
Also by $\ell_p$, we denote the space of all $p$--absolutely summable sequences, where $1\leq
p<\infty$. Moreover, we write $bs,~cs$ and $cs_0$ for the sequences spaces
of all bounded, convergent and null series, respectively.

Let $\mu$ and $\nu$ be two sequence spaces, and let
$A=(a_{nk})$ be an infinite matrix of complex numbers $a_{nk}$,
where $n,k\in\mathbb{N}$. Then we say that $A$ defines a matrix transformation from $\mu$ into
$\nu$, and we denote it by writing $A:\mu\to \nu$  if for every sequence $x=(x_k)\in\mu$, the sequence $Ax=\{(Ax)\}$, the $A$- transform of $x$,
is in $\nu$, where
\begin{eqnarray}\label{A-trans}
(Ax)_{n}:=\sum_k a_{nk}x_k~~,~ (n\in \mathbb{N},~~x\in
D_{00}(A)),
\end{eqnarray}
and by $D_{00}(A)$ denotes the subspace of $w$ consisting of $x\in w$
for which the sum exists as a finite sum.  For simplicity in
notation, here and in what follows, the summation without limits
runs from $0$ to $\infty$ and we shall use the convention that any term with a negative subscript is equal to naught, for example, $\lambda_{-1}=0$ and $x_{-1}=0$.

By $(\mu:\nu)$, we denote the class of all matrices $A$  such that $A:\mu\to \nu$. Thus $A\in(\mu:\nu)$ if and only if the series on
the right side of (\ref{A-trans}) convergens for each $n\in\mathbb{N}$ and each $x\in\mu$ and we have $Ax=\{(Ax)_n\}_{n\in\mathbb{N}}\in\nu$
for all $x\in\mu$ . For an arbitrary sequence space $\mu$, the \emph{matrix domain} $\mu_A$ of an infinite matrix $A$ in
$\mu$ is defined by
 \begin{eqnarray}\label{Matrix domain}
\mu_A:=\{x\in w: Ax\in\mu\},
 \end{eqnarray}
which is a sequence space. If A is triangle, then one can easily observe that the normed sequence spaces $\mu_A$ and $\mu$ are norm isomorphic, i.e., $\mu_A\cong \mu .$ If $\mu$ is a sequence space, then the continuous dual $\mu_{A}^{*}$ of the space $\mu_{A}$ is defined by
\begin{equation*}
    \mu_{A}^{*}:=\{f : f = g\circ A,~~ g \in \mu^{*}\}.
\end{equation*}
We denote the collection of all finite subsets of $\mathbb{N}$ by  $\mathcal{F} $. Also, we will write $e^{(k)}$ for the sequence whose only non-zero term is  $1$ in the $k^{th}$ place for each $k\in \mathbb{N}$. Throughout this paper, let $\lambda=(\lambda_k)$ be a strictly increasing sequence of positive real tending
to infinity; that is
\begin{eqnarray*}
 0<\lambda_0<\lambda_1<\lambda_2<...,~~~~~~~~~\lim_{k\to\infty}\lambda_k=\infty.
 \end{eqnarray*}
We define the matrix $\Lambda=(\lambda_{nk})$ of weighted mean relative to the sequence $\lambda$  by
\begin{eqnarray*}
   \lambda_{nk}=\left\{\begin{array}{ll}
    \frac{\lambda_k-\lambda_{k-1}}{\lambda_n}, & ~~0\leq k\leq n,\\
    0, & ~~k>n,\\
\end{array}
\right.
\end{eqnarray*}
 for all $k,~n\in\mathbb{N}$. With a direct calculation we derive the equality
 \begin{equation}\label{Lamda trans}
    \Lambda_{n}(x)=\frac{1}{\lambda_{n}}\sum_{k=0}^{n}(\lambda_{k}-\lambda_{k-1})x_{k};~~~(n\in\mathbb{N}).
 \end{equation}
 It is easy to show that the matrix $\Lambda$ is regular and is reduced, in the special case
 $\lambda_k=k+1$ for all $k\in\mathbb{N}$ to the matrix $C_1$ of Ces\`{a}ro  means of order one. Introducing the concept
 of $\Lambda-$ \emph{strong convergence}, several results on $\Lambda-$ strong convergence of numerical sequences and Fourier series were given by M\'{o}ricz \cite{Moriz}. Since we have
\begin{eqnarray*}
 Q_n=\sum_{k=0}^{n}q_k=\lambda_n,~~~~~~~~~~~~r_{nk}=\frac{q_k}{Q_n}=\frac{\lambda_k-\lambda_{k-1}}{\lambda_n}=\lambda_{nk}
 \end{eqnarray*}
in the special case $q_k=\lambda_k-\lambda_{k-1}$ for all $k\in\mathbb{N}$, the matrix $\Lambda$ is also reduced to the Riesz means $R^q=(r_{nk})$ with respect to the sequence $q=(q_k)$.

We summarize the knowledge in the existing literature  concerning
with the $\lambda$- matrices domain over some sequence spaces.
Mursaleen and Noman \cite{mn01,mn03,mn04I,mn04II} introduced the spaces $\ell^{\lambda}_{\infty}$, $c^{\lambda}$, $c^{\lambda}_{0}$ and $\ell^{\lambda}_{p}$ of lambda-bounded, lambda-convergent, lambda-null and lambda-absolutely $p-$summable sequences and gave the inclusion relations between these spaces and the classical sequence spaces $\ell_{\infty}$, $c$ and $c_{0}$. Later, Mursaleen and Noman \cite{mn02} investigated the difference spaces $c_{0}^{\lambda}(\Delta)$ and $c^{\lambda}(\Delta)$ obtained from the spaces $c^{\lambda}_{0}$ and $c^{\lambda}$. Recently, on paranormed $\lambda-$ sequence spaces of non-absolute type has been studied by Karakaya, Noman and Polat \cite{vatan}. More recently, S\"onmez and Ba\c sar \cite{acfb} introduce the difference sequence spaces $c_{0}^{\lambda}(B)$ and $c^{\lambda}(B)$, which are the generalization of the spaces $c_{0}^{\lambda}(\Delta)$ and $c^{\lambda}(\Delta)$.  Quite recently, on some new sequence spaces of non-absolute type and matrix transformations has been studied by Ganie and Sheikh \cite{hamit}. Same authors has been studied on spaces of $\lambda-$ convergent sequences and almost convergence \cite{hamit1}. Also, on the fine  spectrum of the operator defined by lambda matrix over the spaces of null and convergent sequences has been studied by Ye\c silkayagil and Ba\c sar \cite{myfb}.

In this work, our purpose is to construct sequence spaces of $cs^{\lambda},~cs_0^{\lambda}$  and $bs^{\lambda}$ by using a matrix domain over a normed space as done by \cite{mn01}.

We define the sequence $y=(y_{k})$, which will be frequently used, as the $\Lambda$- transform of a sequence $x=(x_{k})$, i.e., $y=\Lambda(x)$ and so we have
\begin{equation}\label{18}
   y_{k}:=\frac{1}{\lambda_{k}}\sum_{j=0}^{k}(\lambda_{j}-\lambda_{j-1})x_{j};~~~(k\in\mathbb{N}).
\end{equation}
Also, we say that a sequence $x=(x_k)\in w$ is $\lambda$-convergent if $\Lambda x\in c$. In particular, we say that $x$ is $\lambda$-null sequence if $\Lambda x\in c_0$ and we say that $x$ is $\lambda$-bounded if $\Lambda x\in \ell_{\infty}$.

\section{The sequence spaces $cs^{\lambda},~cs_0^{\lambda}$  and $bs^{\lambda}$}
In the present section, we introduce the sequence spaces $cs^{\lambda},~cs_0^{\lambda}$  and $bs^{\lambda}$ as the sets of all sequences whose $\Lambda$-transforms  are in the spaces $cs,~cs_0$  and $bs$, respectively, that is
\begin{eqnarray*}
  cs^{\lambda} &=& \left\{x=(x_k)\in
w:\lim_{m\to\infty}\sum_{n=0}^{m}\frac{1}{\lambda_{n}}\sum_{k=0}^{n}(\lambda_{k}-\lambda_{k-1})x_{k}~~exists\right\},\\
  cs_0^{\lambda} &=& \left\{x=(x_k)\in
w:\lim_{m\to\infty}\sum_{n=0}^{m}\frac{1}{\lambda_{n}}\sum_{k=0}^{n}(\lambda_{k}-\lambda_{k-1})x_{k}=0 \right\}
\end{eqnarray*}
and
\begin{eqnarray*}
bs^{\lambda}  &=& \left\{x=(x_k)\in
w:\sup_{m}\left|\sum_{n=0}^{m}\frac{1}{\lambda_{n}}\sum_{k=0}^{n}(\lambda_{k}-\lambda_{k-1})x_{k}\right|<\infty\right\}.
\end{eqnarray*}

  With the notation of \eqref{Matrix domain}, we can redefine the spaces $cs^{\lambda},~cs_0^{\lambda}$  and $bs^{\lambda}$ as the matrix domains of the triangle $\Lambda$ in the spaces $ cs, cs_0$ and $bs$  by
  \begin{equation}\label{matrix domain}
     cs^{\lambda}=(cs)_{\Lambda},~ cs_0^{\lambda}=(cs_0)_{\Lambda}~~ \textrm{and}~~ bs^{\lambda}=(bs)_{\Lambda}.
  \end{equation}

Then, it is immediate  by  \eqref{matrix domain} that the sets $  cs^{\lambda},~cs_0^{\lambda}$ \textrm{and} $bs^{\lambda}$ are linear spaces with coordinatewise addition and scalar multiplication, that is, $  cs^{\lambda},~cs_0^{\lambda}$ \textrm{and} $bs^{\lambda}$ are the sequence spaces consisting of all sequences which are $\lambda$-convergent, $\lambda$-null and $\lambda$-bounded series of type $\lambda$, respectively.

Now, we may begin with the following theorem which is essential in the text.
\begin{thm}\label{21}The sequence spaces $ cs^{\lambda},~cs_0^{\lambda}$ and $bs^{\lambda}$ are BK-spaces with the same norm
$ \|x\|_{{cs}^{\lambda}}=\|x\|_{{cs}_0^{\lambda}}=\|x\|_{{bs}^{\lambda}}$, that is,
\begin{equation*}
    \|x\|_{{bs}^{\lambda}}=\|\Lambda(x)\|_{bs}=\sup_{m}\left|\sum_{n=0}^{m}\Lambda_{n}(x)\right|<\infty.
\end{equation*}
\end{thm}
\begin{proof} Since \eqref{matrix domain} holds and $ cs,~cs_0$ and $bs$ are BK- spaces with the respect to their natural norms and the matrix $\Lambda$ is a triangle, Theorem 4.3.12 of Wilansky \cite[page 63]{wilanski} gives the fact that $ cs^{\lambda},~cs_0^{\lambda}$ and $bs^{\lambda}$ are BK- spaces with the given norms. This completes the proof.
\end{proof}
\begin{rem}One can easily check that the absolute property does not hold on the spaces $ cs^{\lambda},~cs_0^{\lambda}$ and $bs^{\lambda}$, that is,
$\|x\|_{{cs}^{\lambda}}\neq\||x|\|_{{cs}^{\lambda}},~~\|x\|_{cs_0^{\lambda}}\neq\||x|\|_{cs_0^{\lambda}}$ and $\|x\|_{{bs}^{\lambda}}\neq\||x|\|_{{bs}^{\lambda}}$ for at least one sequence in the spaces  $cs^{\lambda},~ cs_0^{\lambda}$ and $bs^{\lambda}$, and this shows that $cs^{\lambda},~ cs_0^{\lambda}$ and $bs^{\lambda}$ are the sequence spaces of non-absolute type, where $|x|=(|x_{k}|)$.
\end{rem}

Now, we give the final theorem of this section.
\begin{thm}\label{izomorfizm} The sequence spaces $cs^{\lambda}, ~~cs_0^{\lambda}$ and $bs^{\lambda}$ of non-absolute type are isometrically isomorphic to the spaces $cs,~~ cs_0$ and $bs$, respectively, that is  $cs^{\lambda}\cong cs$,\\ $cs_0^{\lambda}\cong cs_0$ and $bs^{\lambda}\cong bs$.
\end{thm}
\begin{proof}To prove this, we should show the existence of an isometric isomorphism between the spaces $cs_0^{\lambda}$ and $cs_0$. Consider the transformation $T$ defined, with the notation of (\ref{18}), from $cs_0^{\lambda}$ to $cs_0$ by $x\mapsto y(\lambda)=Tx$. Then, $T(x)=y=\Lambda(x)\in cs_0$ for every $x\in cs_0^{\lambda}$ and the linearity of $T$ is clear. Also, it is trivial that $x=\theta$ whenever $Tx=\theta$ and hence $T$ is injective.
Furthermore, let $y=(y_{k})\in cs_0$ be given and define the sequence $x=(x_{k})$ by
\begin{equation}\label{22}
    x_{k}:=\sum_{j=k-1}^{k}(-1)^{k-j}\frac{\lambda_{j}}{\lambda_{k}-\lambda_{k-1}}y_{j};~~~~~~(k\in\mathbb{N}).
\end{equation}
Then, by using (\ref{Lamda trans}) and (\ref{22}), we have for every $n\in\mathbb{N}$ that
\begin{equation*}
\begin{split}
  \Lambda_{n}(x) &= \frac{1}{\lambda_{n}}\sum_{k=0}^{n}(\lambda_{k}-\lambda_{k-1})x_{k} \\
   &= \frac{1}{\lambda_{n}}\sum_{k=0}^{n}\sum_{j=k-1}^{k}(-1)^{k-j}\lambda_{j}y_{j} \\
   &= \frac{1}{\lambda_{n}}\sum_{k=0}^{n}(\lambda_{k}y_{k}-\lambda_{k-1}y_{k-1}) \\
   &=y_{n}.
 \end{split}
\end{equation*}
This shows that $\Lambda(x)=y$ and since $y\in cs_0$, we obtain that $\Lambda(x)\in cs_0$. Thus, we deduce that $x\in cs_0^{\lambda}$ and $Tx=y$. Hence $T$ is surjective. Moreover, one can easily see for every $x\in cs_0^{\lambda}$ that
\begin{eqnarray*}
  \|Tx\|_{cs_0} &=& \|y(\lambda)\|_{cs_0} = \|\lambda(x)\|_{cs_0} = \|x\|_{cs_0^{\lambda}}
  \end{eqnarray*}
  which means that $T$ is norm preserving. Therefore $T$ is isometry. Consequently $T$ is an isometric isomorphism which show that the spaces $cs_0^{\lambda}$ and $cs_0$ are  isometrically isomorphic.

  It is clear that if the spaces  $cs_0^{\lambda}$ and $cs_0$ are replaced by the respective one of the spaces $cs^{\lambda}$ and $cs$ or $bs^{\lambda}$ and $bs$, then we obtain the fact that $cs^{\lambda}\cong cs$ and $bs^{\lambda}\cong bs$. This completes the proof.
\end{proof}
\section{The inclusion Relations}
In the present section, we establish some inclusion relations concerning with the  spaces $cs^{\lambda},~cs_0^{\lambda}$  and $bs^{\lambda}$. We may begin with the following lemma.
\begin{lem} For any sequence $x=(x_{k})\in w$, the equality
\begin{equation}\label {31}
    S_{n}(x)=x_{n}-\Lambda_{n}(x);~~~~~~(n\in \mathbb{N})
\end{equation}
holds, where $S(x)=\{S_{n}(x)\}$ is the sequence defined by
$$S_{0}(x)=0 ~~~~~~~\textrm{and}~~~~~~~~ S_{n}(x)=\frac{1}{\lambda_{n}}\sum_{k=1}^{n}\lambda_{k-1}(x_{k}-x_{k-1});~~~~(n\geq1).$$
\end{lem}
\begin{thm}The inclusions $cs_0^{\lambda}\subset cs^{\lambda}\subset bs^{\lambda}$ strictly hold.
\end{thm}
\begin{proof}It is obvious that the inclusions $cs_0^{\lambda}\subset cs^{\lambda}\subset bs^{\lambda}$ hold.

 Let us consider the sequence $x=(x_k)$ defined by
$$x_k=\frac{\lambda_{k}\frac{1}{(k+2)^2}-\lambda_{k-1}\frac{1}{(k+1)^2}}{\lambda_{k}-\lambda_{k-1}};~~(k\in \mathbb{N}).$$
In the present case, we obtain for every $n\in \mathbb{N}$ that equality
$$ \Lambda_{n}(x)=\frac{1}{\lambda_{n}}\sum_{k=0}^{n}\left(\lambda_{k}\frac{1}{(k+2)^2}-\lambda_{k-1}\frac{1}{(k+1)^2}\right)=\frac{1}{(n+2)^2}$$
which shows that $\Lambda(x)\in cs\backslash cs_0.$ Thus, the sequence $x$ is in $cs^{\lambda}$ but not in $cs_{0}^{\lambda}.$ Hence $cs_0^{\lambda}\subset cs^{\lambda}$ is a strict inclusion.

 To show that $cs^{\lambda}\subset bs^{\lambda}$ inclusion is strict, we define the sequence $y=(y_{k})$ by
$$y_{k}=(-1)^{k}\left(\frac{\lambda_{k}+\lambda_{k-1}}{\lambda_{k}-\lambda_{k-1}}\right);~~~~~(k\in\mathbb{N}).$$
Then, we have for every $n\in \mathbb{N}$ that
$$\sum_{n=0}^{m}\Lambda_{n}(y)=\sum_{n=0}^{m}\frac{1}{\lambda_{n}}\sum_{k=0}^{n}(-1)^{k}(\lambda_{k}+\lambda_{k-1})= \sum_{n=0}^{m}(-1)^{n}.$$
This shows $\Lambda(y)\in bs\backslash cs$. Thus, the sequence $y$ is in $bs^{\lambda}$ but not in $cs^{\lambda}$ and hence  $cs^{\lambda}\subset bs^{\lambda}$ is a strict inclusion. This concludes the proof.
\end{proof}
\begin{lem}\cite[Theorem 4.1.]{mn01}The inclusions $c_{0}^{\lambda}\subset c^{\lambda}\subset \ell_{\infty}^{\lambda}$ strictly hold.
\end{lem}
\begin{thm}The inclusions $cs^{\lambda}\subset c_0^{\lambda}$ and $bs^{\lambda}\subset \ell_{\infty}^{\lambda}$ strictly hold.
\end{thm}
\begin{proof} It is clear that the inclusion $cs^{\lambda}\subset c_0^{\lambda}$  holds, since $x\in cs^{\lambda}$ implies $\Lambda(x)\in cs$ and hence $\Lambda(x)\in c_0$ which means that $x\in c_0^{\lambda}.$ Consider the sequence $x=(x_{k})$ defined by
$$x_{k}=\frac{1}{k+1};~~~~~~(k\in \mathbb{N}).$$
Then, $x\in c_0$ and hence $x\in c_0^{\lambda}$, since the inclusion $c_0\subset c_0^{\lambda}$ holds. On the other hand, we have for every $n\in \mathbb{N}$ that
\begin{eqnarray*}
\Lambda_{n}(x) &=& \frac{1}{\lambda_{n}}\sum_{k=0}^{n}\frac{\lambda_{k}-\lambda_{k-1}}{k+1} \\
 &\geq& \frac{1}{\lambda_{n}(n+1)}\sum_{k=0}^{n}(\lambda_{k}-\lambda_{k-1}) \\
 &=& \frac{1}{n+1}
\end{eqnarray*}
which shows that $\Lambda(x)\not\in cs$ and hence $x\not\in cs^{\lambda}.$ Thus, the sequence $x$ is in $c_0^\lambda$ but not in $cs^{\lambda}.$ Therefore, the inclusion $cs^{\lambda}\subset c_0^{\lambda}$ is strict.

Similarly, it is also trivial that the inclusion $bs^{\lambda}\subset \ell_{\infty}^{\lambda}$ holds. To show that this inclusion is strict, we define the sequence $y=(y_{k})$ by
$$y_{k}=e=(1, 1, 1,...)~;~~~~ {(k\in\mathbb{N})}.$$
In the present case, we have for every $n\in\mathbb{N}$ that
$$\Lambda_{n}(y)=\frac{1}{\lambda_{n}}\sum_{k=0}^{n}(\lambda_{k}-\lambda_{k-1})=1$$
which shows that $\Lambda(y)\in \ell_{\infty}\backslash bs$. Thus, the sequence $y$ is in $\ell_{\infty}^{\lambda}$ but not in $bs^{\lambda}$ and hence $bs^{\lambda}\subset \ell_{\infty}^{\lambda}$ is a strict inclusion. This completes the proof.

\end{proof}
\begin{thm} The inclusion $cs^{\lambda}\subset cs$ holds if and only if $S(x)\in cs$ for every sequence $x\in cs^{\lambda}.$
\end{thm}
\begin{proof}Suppose that the inclusion $cs^{\lambda}\subset cs$ holds, and take any $x=(x_{k})\in cs^{\lambda}$. Then $\Lambda x\in cs$ and $x\in cs$ by the hypothesis. Thus, we deduce from \eqref{31} that
\begin{equation}\label{33}
    \sum_{n=0}^{m}(x_{n}-\Lambda_{n}(x))=\sum_{n=0}^{m}x_{n}-\sum_{n=0}^{m}\Lambda_{n}(x)=\sum_{n=0}^{m}S_{n}(x).
\end{equation}
Hence, we obtain from \eqref{33} by letting $m\to \infty$ that
\begin{equation}\label{34}
    \lim_{m}\sum_{n=0}^{m}S_{n}(x)=\sum_{n}x_{n}-\sum_{n}\Lambda_{n}(x).
\end{equation}
As $(x_{n})\in cs$ and $(\Lambda_{n}(x))\in cs$, the right hand side of the equality \eqref {34} is convergent  as $m\to \infty $. Thereby, the series $\sum_{n=0}^{m}S_{n}(x)$ converges and so, $S(x)\in cs.$

Conversely, let $x\in cs^{\lambda}$ be given. Then, we have by the hypothesis that\\ $S(x)\in cs.$ Again, it follows by \eqref{33} that
$$\lim_{m}\sum_{n=0}^{m}x_{n}=\sum_{n}S_{n}(x)+\sum_{n}\Lambda_{n}(x)$$
which shows that $x\in cs$ since $\Lambda(x)\in cs$ and $S(x)\in cs.$ Hence, the inclusion $cs^{\lambda}\subset cs$ holds and this concludes the proof.
\end{proof}
\begin{thm} The inclusion $cs_0^{\lambda}\subset cs_0$ holds if and only if $S(x)\in cs_0$ for every sequence $x\in cs_0^{\lambda}.$
\end{thm}
\begin{proof} One can see by analogy that the inclusion  $cs_0^{\lambda}\subset cs_0$ also holds if and only if $S(x)\in cs_0$ for every sequence $x\in cs_0^{\lambda}.$ This completes the proof.
\end{proof}
\begin{thm} The inclusion $bs^{\lambda}\subset bs$ holds if and only if $S(x)\in bs$ for every sequence $x\in bs^{\lambda}.$
\end{thm}
\begin{proof}Suppose that the inclusion $bs^{\lambda}\subset bs$ holds, and take any $x=(x_{k})\in bs^{\lambda}$. Then, $x\in bs$ by the hypothesis. Thus, we obtain from equality \eqref{31}
$$\|S(x)\|_{bs}\leq \|x\|_{bs}+\|\Lambda(x)\|_{bs}=\|x\|_{bs}+\|x\|_{bs^{\lambda}}<\infty$$
which yields that $S(x)\in bs.$

Conversely, assume that  $S(x)\in bs$ for every $x\in bs^\lambda$. Again, we obtain from equality \eqref{31}
$$\|x\|_{bs}\leq \|S(x)\|_{bs}+\|\Lambda(x)\|_{bs}=\|S(x)\|_{bs}+\|x\|_{bs^{\lambda}}<\infty.$$
This shows that $x\in bs$. Hence, the inclusion $bs^{\lambda}\subset bs$ holds. This completes the proof.
\end{proof}

\section{The Basis for the Spaces $cs^{\lambda},~~ cs_0^{\lambda}$ and $bs^{\lambda}$}
In the present section, we give a sequence of the points of the spaces $cs^{\lambda}$ and $~cs_0^{\lambda}$  which forms a basis for these spaces.
 If a normed sequence space X contains a sequence $(b_n)$ with the property that for every $x\in X$ there is a unique sequence of scalars $(\alpha_n)$ such that
\begin{equation*}
\lim_{n\to\infty}\|x-(\alpha_0b_0+\alpha_1b_1+\ldots+\alpha_nb_n)\|=0
\end{equation*}
then $(b_n)$ is called a Schauder basis (or briefly basis) for $X$. The series $\sum_k \alpha_kb_k$ which has the sum $x$ is then called the expansion of $x$ with respect to $(b_n)$ and is written as $x=\sum_k \alpha_kb_k$.

Now, since the transformation $T$ defined from $cs_0^{\lambda}$ to $cs_0$ in the proof of Theorem \ref{izomorfizm} is an isomorphism. Therefore, we have the following theorem:
\begin{thm}\label{41}
Define the sequence  $e_{\lambda}^{(n)}=\{(e_{\lambda}^{(n)})_k\}$  for every fixed $k\in\mathbb{N}$ by
\begin{eqnarray*}
   \{(e_{\lambda}^{(n)})_k\}=\left\{\begin{array}{ll}
    (-1)^{k-n}\frac{\lambda_n}{\lambda_{k}-\lambda_{k-1}}&,  ~~n\leq k\leq n+1,\\
    0 &,  ~~\textrm{otherwise},\\
\end{array}
\right.
\end{eqnarray*}
 for all $k\in\mathbb{N}$.

 Then, we have\\
 a) The sequence $(e_{\lambda}^{(0)},~ e_{\lambda}^{(1)},\ldots)$ is a Schauder basis  for the spaces $cs^{\lambda}$ and  $cs_0^{\lambda}$ and every $x\in cs^{\lambda}~~ \textrm{or}~~cs_0^{\lambda}$ has a unique representation of the form
 \begin{equation}\label{43}
    x=\sum_{n=0}^{\infty}\Lambda_{n}(x)e_\lambda^{(n)}.
 \end{equation}
 b) $bs^{\lambda}$ has no Schauder basis.
\end{thm}
\begin{proof} a) ([13, Theorem 2.3.]) It is clear that $e_{\lambda}^{(n)}$ is a basis for $cs^{\lambda}$  since $e^{(n)}$ is a basis for $cs$  and $\Lambda(e_{\lambda}^{(n)})=e^{(n)}.$ Let $x\in cs^{\lambda}$ be given. Then,  $y=\Lambda(x)\in cs $ and $$y^{[m]}=\sum_{n=0}^my_ne^{(n)}\to y~~(m\to\infty)$$ for a unique sequence $(y_{n})_{n=0}^{\infty}$ of scalars. Therefore, we obtain  that
$$\Lambda^{-1}(y^{[m]})=\sum_{n=0}^my_n\Lambda^{-1}(e^{(n)})=\sum_{n=0}^my_ne_\lambda^{(n)}.$$
Since $y_{n}=\Lambda_{n}(x)$, we can write as
$$x^{[m]}=\sum_{n=0}^m\lambda_n(x)e_\lambda^{(n)}.$$ Consequently,
$$\|x^{[m]}-x\|_{cs^{\lambda}}=\|\Lambda(x^{[m]}-x)\|_{cs}=\|\Lambda(x^{[m]})-\Lambda(x)\|_{cs}=\|y^{[m]}-y\|_{cs}\to 0~~(m\to\infty). $$
Thus, we deduce that $\lim_{m\rightarrow\infty}\|x^{[m]}-x\|=0$, which shows that $x\in cs^{\lambda}$ is represented as in \eqref{43}.\\

Finally, let us show the uniqueness of the representation \eqref{43} of $x\in cs^{\lambda}.$ For, suppose on the contrary that there exists another representation $ x=\sum_{n}\alpha_{n}e_\lambda^{(n)}.$ Since the linear transformation $T$ defined from $cs^{\lambda}$ to $cs$, in the proof of Theorem \ref{izomorfizm}, is continuous, we have
$$\Lambda_{k}(x)=\sum_{n}\alpha_{n}\Lambda_{k}(e_\lambda^{(n)})=\sum_{n}\alpha_{n}\delta_{kn}=\alpha_{k};~~~~(k\in\mathbb{N}).$$ Therefore, the representation \eqref{43} of $x\in cs^{\lambda}$ is unique.  It can be proved similarly for $cs_0^{\lambda}$. This completes the proof.

b)As a direct consequence of Remark 2.2. of  Malkowsky and  Rako\u{c}evi\'{c} \cite{malkowskiv}, $bs^{\lambda}$ has no Schauder basis.
\end{proof}

As a result, it easily follows from Theorem \ref{21} that $cs^{\lambda}$ and $cs_0^{\lambda}$ are the Banach spaces with their natural norms. Then by Theorem \ref{41} we obtain the the following corollary:

\begin{cor} The sequence spaces $cs^{\lambda}$ and $cs_0^{\lambda}$ of non-absolute type are separable.
\end{cor}
\section{The $\alpha-,~\beta-$  and $\gamma-$ duals  of the spaces $cs^{\lambda}, cs_0^{\lambda}$ and $bs^{\lambda}$ }
In this section, we state and prove the theorems determining the $\alpha-, \beta-$ and $\gamma-$
duals of the sequence spaces $cs^{\lambda}, cs_0^{\lambda}$  and $bs^{\lambda}$ of non-absolute type. For arbitrary sequence spaces $X$ and $Y$, the set $M(X,Y)$ defined by
\begin{equation}\label{dual}
M(X,Y)=\{a=(a_{k})\in w: ax=(a_{k}x_{k})\in Y~~ \forall x=(x_{k})\in X\}
\end{equation}
is called the $multiplier ~~~~space$ of $X$ and $Y$. One can easily observe for a sequence space $Z$ with $Y\subset Z$ and $Z\subset X$ that the inclusions              $ M(X,Y)\subset M(X,Z)$ and $M(X,Y)\subset M(Z,Y)$ hold, respectively. With the notation of (\ref{dual}), the $\alpha-,~\beta-$ and $\gamma-$ duals  of a sequence space $X$, which are respectively, denoted by $X^{\alpha},~ X^{\beta}$ and $X^{\gamma}$ are defined by
\begin{eqnarray*}
  X^{\alpha} = M(X,\ell_{1}),~~~~~~ X^{\beta} = M(X,cs)~\textrm{and}~ X^{\gamma} = M(X,bs).
\end{eqnarray*}\\
It is clear that $X^{\alpha}\subset X^{\beta}\subset X^{\gamma}$. Also, it can be obviously seen that the inclusions $X^{\alpha}\subset Y^{\alpha}$,  $X^{\beta}\subset Y^{\beta}$, and $X^{\gamma}\subset Y^{\gamma}$ hold whenever $Y\subset X$.

The following known results \cite{st} are fundamental for this section.

\begin{lem}\label{cs l}$A=(a_{nk})\in (cs:\ell_{1})$ if and only if
\begin{equation}
\sup_{N,K\in\mathcal{F}}\left|\sum_{n\in N}\sum_{k\in K}(a_{nk}-a_{n,k-1})\right|<\infty.
\end{equation}
\end{lem}
\begin{lem}$A=(a_{nk})\in (cs_0:\ell_{1})$ if and only if
\begin{equation}\label{53}
\sup_{N,K\in\mathcal{F}}\left|\sum_{n\in N}\sum_{k\in K}(a_{nk}-a_{n,k+1})\right|<\infty.
\end{equation}
\end{lem}
\begin{lem}\label{bs l}$A=(a_{nk})\in (bs:\ell_{1})$ if and only if \eqref{53} holds and
\begin{equation}\label{55}
    \lim_{k}a_{nk} = 0,~~~~\forall n\in\mathbb{N}.
\end{equation}

\end{lem}

\begin{lem}\label{cs c}$A=(a_{nk})\in (cs:c)$ if and only if
\begin{equation}\label{56}
    \sup_{n}\sum_{k}|a_{nk}-a_{n,k+1}|<\infty,
\end{equation}
and
\begin{equation}\label{58}
\lim_{n}a_{nk}~~~~~~~~~~ \textrm {exist for all} ~~~~~~k\in \mathbb{N}.
\end{equation}
\end{lem}
\begin{lem}\label{cs_0 c}$A=(a_{nk})\in (cs_0:c)$ if and only if \eqref{56} holds
\begin{equation}
    \lim_{n}(a_{nk}-a_{n,k+1})~~~~~~~~~~ \textrm {exist for all}~~~~ k\in \mathbb{N}.
\end{equation}
\end{lem}
\begin{lem}\label{228}$A=(a_{nk})\in (bs:c)$ if and only if  \eqref{55} and \eqref {58} hold and
\begin{equation}
\sum_{k}|a_{nk}-a_{n,k-1}|~~~~\textrm converges.
\end{equation}
\end{lem}

\begin{lem}\label{cs gama}$A=(a_{nk})\in (cs:\ell_{\infty})$ if and only if
\begin{equation}
    \sup_{n}\sum_{k}|a_{nk}-a_{n,k-1}|<\infty.
\end{equation}
\end{lem}
\begin{lem}\label{cs 0 gama}$A=(a_{nk})\in (cs_0:\ell_{\infty})$ if and only if  \eqref {56} holds.
\end{lem}
\begin{lem}\label{bs gama}$A=(a_{nk})\in (bs:\ell_{\infty})$ if and only if \eqref{55} and \eqref{56} hold.
\end{lem}

Now, we prove the following result.
\begin{thm} Define the sets  $m_{1}^{\lambda}$ and $m_{2}^{\lambda}$ as follows:
\begin{equation}\label{m bir }
    m_{1}^{\lambda}=\left\{a=(a_n)\in w:\sup_{N,K\in\mathcal{F}}\left|\sum_{n\in N}\sum_{k\in K}(b_{nk}^{\lambda}-b_{n,k-1}^{\lambda})\right|<\infty\right\}
\end{equation}
and
\begin{equation}\label{m iki }
    m_{2}^{\lambda}=\left\{a=(a_n)\in w:\sup_{N,K\in\mathcal{F}}\left|\sum_{n\in N}\sum_{k\in K}(b_{nk}^{\lambda}-b_{n,k+1}^{\lambda})\right|<\infty\right\};
\end{equation}
where the matrix $B^{\lambda}=(b_{nk}^{\lambda})$ is defined via the sequence $a=(a_{n})\in w$ by
\begin{eqnarray*}
   b_{nk}^{\lambda}=\left\{\begin{array}{ll}
    (-1)^{n-k}\frac{\lambda_{k}}{\lambda_n-\lambda_n-1}{a_{n}} &~~~~~\textrm{if}~~~~~~~~n-1\leq k\leq n,\\
    0  &~~~~\textrm{if}~~~~~~  ~~~~n-1>k ~~~~~\textrm{or}~~~~  k>n,\\
\end{array}
\right.
\end{eqnarray*}
for all $n,~k\in\mathbb{N}$.
Then $\{cs^{\lambda}\}^{\alpha}=m_{1}^{\lambda}$ and $\{cs_0^{\lambda}\}^{\alpha}=\{bs^{\lambda}\}^{\alpha}=m_{2}^{\lambda}$.

\end{thm}
\begin{proof}Let $a=(a_{n})\in w$. Then, by bearing in mind the relations  (\ref{18}) and (\ref{22}), it is immediate that the equality
\begin{equation}\label{514}
    a_{n}x_{n}=\sum_{k=n-1}^{n}(-1)^{n-k}\frac{\lambda_{k}}{\lambda_{n}-\lambda_{n-1}}{a_{n}y_{k}}=B_{n}^{\lambda}(y)
\end{equation}
holds for all $n\in \mathbb{N}$. We therefore observe by \eqref{514} that $ax=(a_{n}x_{n})\in\ell_{1}$ whenever $x=(x_{k})\in cs^{\lambda}$ if and only if  $B^{\lambda}y\in \ell_{1}$ whenever $y=(y_{k})\in cs$. This means that the sequence $a=(a_{n})\in \{cs^{\lambda}\}^{\alpha}$ if and only if        $B^{\lambda}\in (cs:\ell_{1})$. Hence, we obtain by Lemma \ref{cs l} with $B^{\lambda}$ instead of $A$ that $a=(a_{n})\in \{cs^{\lambda}\}^{\alpha}$ if and only if

\begin{equation*}
    \sup_{N,K\in\mathcal{F}}\left|\sum_{n\in N}\sum_{k\in K}(b_{nk}^{\lambda}-b_{n,k-1}^{\lambda})\right|<\infty
\end{equation*}
which yields the result that $\{cs^{\lambda}\}^{\alpha}=m_{1}^{\lambda}$.

Similarly, we deduce from  Lemma \ref{bs l} with (\ref{514}) that $a=(a_{n})\in \{bs^{\lambda}\}^{\alpha}$ if and only if $B^{\lambda}\in (bs:\ell_{1})$. Then, it is clear that the columns of the matrix $B$ are in the space $c_0$, since
\begin{equation*}
    \lim_{n}b_{nk}^{\lambda}=0
\end{equation*}
for all $k\in \mathbb{N}$. Therefore, we derive from \eqref{53} that
\begin{equation}\label{cs_0 alfa}
\sup_{N,K\in\mathcal{F}}\left|\sum_{n\in N}\sum_{k\in K}(b_{nk}^{\lambda}-b_{n,k+1}^{\lambda})\right|<\infty.
\end{equation}
This shows that $\{cs_0^{\lambda}\}^{\alpha}=\{bs^{\lambda}\}^{\alpha}=m_{2}^{\lambda}$. This completes the proof.
\end{proof}
\begin{thm}\label{311}
Define the sets $m_{3}^{\lambda}$ and  $m_{4}^{\lambda}$ as follows :
\begin{eqnarray*}
  m_{3}^{\lambda}  &=& \left\{a=(a_k)\in w:\sum_{k=0}^{\infty}\left|\bar{\Delta}\left(\bar{\Delta}(\frac{a_{k}}{\lambda_{k}-\lambda_{k-1}}){\lambda_{k}}\right) \right|<\infty\right\}, \\ \\
  m_{4}^{\lambda} &=& \left\{a=(a_k)\in w:\sup_{k}\left|\frac{\lambda_{k}}{\lambda_{k}-\lambda_{k-1}}{a_{k}}\right|<\infty\right\}, \\
  m_{5}^{\lambda}&=& \left\{a=(a_k)\in w:\lim_{k\rightarrow \infty}\left|\frac{\lambda_{k}}{\lambda_{k}-\lambda_{k-1}}{a_{k}}\right|~~\textrm{exists}\right\},
\end{eqnarray*}

where
\begin{equation*}
    \bar{\Delta}\left(\frac{a_{k}}{\lambda_{k}-\lambda_{k-1}}\right)=\frac{a_{k}}{\lambda_{k}-\lambda_{k-1}}-\frac{a_{k+1}}{\lambda_{k+1}-\lambda_{k}}
    \end{equation*}
for all $k\in \mathbb{N}$. Then $\{cs^{\lambda}\}^{\beta}=\{cs_0^{\lambda}\}^{\beta}=m_{3}^{\lambda}\cap m_{4}^{\lambda}$ and $\{bs^{\lambda}\}^{\beta}=m_{3}^{\lambda}\cap m_{5}^{\lambda}$.
\end{thm}
\begin{proof}Because of the proof may also be obtained for the space $bs^{\lambda}$ in the similar way, we omit it. Take any $a=(a_{k})\in w$ and consider the equation
\begin{eqnarray}\label{T y}
\begin{split}
  \sum_{k=0}^{n}a_{k}x_{k} &= \sum_{k=0}^{n}\left[\sum_{j=k-1}^{k}(-1)^{k-j}\frac{\lambda_{j}}{\lambda_{k}-\lambda_{k-1}}{y_{j}}\right]a_{k}\\
   &=\sum_{k=0}^{n-1}\bar{\Delta}\left(\frac{a_{k}}{\lambda_{k}-\lambda_{k-1}}\right)\lambda_{k}y_{k}+\frac{\lambda_{n}}{\lambda_{n}-\lambda_{n-1}}a_{n}y_{n}=T_{n}^{\lambda}(y),
\end{split}
\end{eqnarray}
where the matrix  $T^{\lambda}=(t_{nk}^{\lambda})$ is defined  by
\begin{eqnarray*}
   t_{nk}^{\lambda}=\left\{\begin{array}{ll}
    \bar{\Delta}\left(\frac{a_{k}}{\lambda_{k}-\lambda_{k-1}}\right)\lambda_{k} &~~~~\textrm {if}  ~~0\leq k\leq n-1,\\
    \frac{\lambda_{n}}{\lambda_{n}-\lambda_{n-1}}a_{n} & ~~~~\textrm {if}  ~~k=n,\\
    0 &~~~~\textrm {if}  ~~ k>n.\\
\end{array}
\right.
\end{eqnarray*}
for all $n,~k\in\mathbb{N}$. Thus, we deduce by \eqref{T y} that $ax=(a_{k}x_{k})\in cs$ where $x=(x_{k})\in cs^{\lambda}$ if and only if $T^{\lambda}(y)\in c$ whenever $y=(y_{k})\in cs$. This means that $a=(a_{k})\in \{cs^{\lambda}\}^{\beta}$ if and only if $T^{\lambda}\in (cs:c)$. Therefore, by using Lemma \ref{cs c}, we derive from        \eqref{56} and \eqref{58} that
\begin{equation*}
    \sum_{k=0}^{\infty}\left|\bar{\Delta}\left(\bar{\Delta}(\frac{a_{k}}{\lambda_{k}-\lambda_{k-1}}){\lambda_{k}}\right)\right|<\infty,
\end{equation*}
\begin{equation*}
    \sup_{n}\left|\frac{\lambda_{n}}{\lambda_{n}-\lambda_{n-1}}{a_{n}}\right|<\infty
\end{equation*}
and
\begin{equation*}
   \lim_{n}t_{nk}^{\lambda}=\bar{\Delta}\left(\frac{a_{k}}{\lambda_{k}-\lambda_{k-1}}\right)\lambda_{k},
\end{equation*}
respectively. Thereby, we conclude that $\{cs^{\lambda}\}^{\beta}=\{cs_0^{\lambda}\}^{\beta}=m_{3}^{\lambda}\cap m_{4}^{\lambda}$.
\end{proof}
\begin{thm}\label{gama} The $\gamma$-dual of the space $cs^{\lambda},~cs_0^{\lambda}$ and $bs^{\lambda}$ is the set $m_{3}^{\lambda}\cap m_{4}^{\lambda}$.
\end{thm}
\begin{proof}The proof of this result follows the same lines that in the proof of Theorem \ref{311} using   Lemma \ref{cs gama}, Lemma \ref{cs 0 gama} and    Lemma \ref{bs gama} instead of Lemma \ref {cs c}.
\end{proof}

\section{ Certain matrix mappings on the spaces $cs^{\lambda}, cs_0^{\lambda}$  and $bs^{\lambda}$}
In this present section, we characterize the matrix classes $(cs^{\lambda}: \ell_{p}),~(cs_0^{\lambda}: \ell_{p}),~\\(bs^{\lambda}: \ell_{p}),~(cs^{\lambda}:c_0),~(cs_0^{\lambda}:c_0),~(bs^{\lambda}:c_0),~(cs^{\lambda}:c),~(cs_0^{\lambda}:c)$ and $(bs^{\lambda}:c)$, where $1\leq p\leq\infty$.

For an infinite matrix $A=(a_{nk})$, we write for brevity that
\begin{equation*}
    \tilde{a}_{nk}=\bar{\Delta}\left(\frac{a_{nk}}{\lambda_{k}-\lambda_{k-1}}\right){\lambda_{k}}=\left(\frac{a_{nk}}{\lambda_{k}-\lambda_{k-1}}-\frac{a_{n,k+1}}{\lambda_{k+1}-\lambda_{k}}\right)\lambda_{k}~~(n,k\in \mathbb{N}).
\end{equation*}

The following lemmas will be needed in proving our  results.
\begin{lem}\label{cs c_0}$A=(a_{nk})\in (cs:c_0)$ if and only if \eqref{56} holds and
\begin{equation}
    \lim_{n}a_{nk} = 0~~~~(\forall k\in\mathbb{N}).
\end{equation}
\end{lem}
\begin{lem}\label{cs_0 c_0}$A=(a_{nk})\in (cs_0:c_0)$ if and only if \eqref{56} holds and
\begin{equation}
    \lim_{n}(a_{nk}-a_{n,k+1})=0~~~~(\forall k\in \mathbb{N}).
\end{equation}
\end{lem}
\begin{lem}$A=(a_{nk})\in (bs:c_0)$ if and only if \eqref{55} holds and
\begin{equation}
\lim_{n} \sum_{k}|a_{nk}-a_{n,k+1}|=0.
\end{equation}
\end{lem}
\begin{lem}$A=(a_{nk})\in (cs_0:\ell_{p})$ if and only if
\begin{equation}\label{cs lp}
    \sup_{k}\sum_{n}\left|\sum_{k\in K}(a_{nk}-a_{n,k+1})\right|^{p}<\infty~~~~~~(1<p<\infty).
\end{equation}
\end{lem}
\begin{lem}$A=(a_{nk})\in (cs:\ell_{p})$ if and only if
\begin{equation}
    \sup_{k}\sum_{n}\left|\sum_{k\in K}(a_{nk}-a_{n,k-1})\right|^{p}<\infty~~~~~~(1<p<\infty).
\end{equation}
\end{lem}

\begin{lem}$A=(a_{nk})\in (bs:\ell_{p})$ if and only if \eqref{55} and \eqref{cs lp} hold.
\end{lem}

Now, we give the following results on the matrix transformations.
\begin{thm}\label{t47}
$(i)$ $A=(a_{nk})\in (cs^{\lambda}:\ell_{\infty})$ if and only if

\begin{equation}\label{47}
    \sup_{n}\sum_{k=0}^{\infty}\left|\tilde{a}_{nk}-\tilde{a}_{k-1} \right|<\infty
\end{equation}
and
\begin{equation}\label{48}
    \sup_{k}\left|\frac{\lambda_{k}}{\lambda_{k}-\lambda_{k-1}}{a_{nk}}\right|<\infty.
\end{equation}
$(ii)$ $A=(a_{nk})\in (cs_{0}^{\lambda}:\ell_{\infty})$ if and only if \eqref{48} holds and

\begin{equation}\label{68}
    \sup_{n}\sum_{k=0}^{\infty}\left|\tilde{a}_{nk}-\tilde{a}_{k+1} \right|<\infty.
\end{equation}
$(iii)$ $A=(a_{nk})\in (bs^{\lambda}:\ell_{\infty})$ if and only if  \eqref{68} holds and
\begin{equation}\label{69}
\lim_{k\rightarrow \infty}\left|\frac{\lambda_{k}}{\lambda_{k}-\lambda_{k-1}}{a_{nk}}\right|~~\textrm{exists},
\end{equation}

\begin{equation}\label{416}
\lim_{k}\tilde{a}_{nk}=0
\end{equation}
and
\begin{equation}
    \sup_{n}|a_{n}|<\infty.
\end{equation}
\end{thm}
\begin{proof}Suppose that conditions \eqref{47} and \eqref{48} hold and take any $x=(x_{k})\in cs^{\lambda}.$ Then, we have by Theorem \ref{311} that $(ank)_{k=0}^{\infty}\in (cs^{\lambda})^{\beta}$ for all $n\in \mathbb{N}$ and this implies the existence of the $A$-transform of $x$, i.e.; Ax exists. Further, it is clear that the associated sequence $y=(y_{k})$ is in the $cs$ and hence $y\in c_0.$

Let us now consider the following equality derived by using the relation (\ref{18}) from the  $m^{th}$ partial sum of the series $\sum_{k}a_{nk}x_{k}:$
\begin{equation}\label{kýsmi toplam}
    \sum_{k=0}^{m}a_{nk}x_{k}=\sum_{k=0}^{m-1}\tilde{a}_{nk}y_{k}+\frac{\lambda_{m}}{\lambda_{m}-\lambda_{m-1}}{a_{nm}y_{m}}, ~~(\forall n,m\in \mathbb{N}).
\end{equation}
Therefore, by using  \eqref{47} and \eqref{48}, from \eqref{kýsmi toplam} as $m\rightarrow\infty$ we obtain that equality
\begin{equation}\label{ax}
   \sum_{k}a_{nk}x_{k}=\sum_{k}\tilde{a}_{nk}y_{k}~~\textrm{for all}~~n\in \mathbb{N}.
\end{equation}
Further, since the matrix $\tilde{A}=(\tilde{a}_{nk})$ is in the class $(cs:\ell_{\infty})$ by Lemma  \ref{cs gama} and \eqref{47}; we have $\tilde{A}y\in\ell_{\infty}.$ Therefore, we deduce from \eqref{A-trans} and \eqref{ax} that $Ax\in \ell_{\infty}$ and hence $A\in (cs^{\lambda},\ell_{\infty}).$

Conversely, suppose that $A\in (cs^{\lambda},\ell_{\infty}).$ Then $(a_{nk})_{k=0}^{\infty}\in (cs^{\lambda})^{\beta} $ for all $n\in \mathbb{N}$ and this, with  Theorem \ref{311}, implies both \eqref{48} and
$$ \sum_{k=0}^{\infty}\left|\tilde{a}_{nk}-\tilde{a}_{k+1} \right|<\infty ~~\textrm {for all}~~~~ n\in \mathbb{N}$$
which together imply that relation \eqref{ax} holds for all sequences $x\in cs^{\lambda}$ and $y\in cs.$ Further, since $Ax\in \ell_{\infty}$ by the hipothesis; we obtain by \eqref{ax} that $\tilde{A}y\in \ell_{\infty}$ which shows that $\tilde{A}\in (cs:\ell_{\infty}),$ where
 $\tilde{A}=(\tilde{a}_{nk})$. Hence, the necessity of \eqref{47} is immediate by \eqref {cs gama}. This concludes the proof of part $(i).$

Since part $(ii)$ can be proved similarly, we omit its proof.
\end{proof}
\begin{cor}\label{t48} $(i)$ $A=(a_{nk})\in(cs^{\lambda}:c)$ if and only if \eqref{48}, and \eqref{68} hold and
\begin{equation}\label{614}
    \lim_{n\rightarrow\infty}\tilde{a}_{nk}~~\textrm{exists}.
\end{equation}
$(ii)$ $A=(a_{nk})\in (cs_0^{\lambda}:c)$ if and only if \eqref{48} and \eqref{68} hold and
\begin{equation}
    \lim_{n}(\tilde{a}_{nk}-\tilde{a}_{n,k+1}) ~~~~~\textrm{exists}.
\end{equation}
$(iii)$ $A=(a_{nk})\in (bs^{\lambda}:c)$ if and only if  \eqref{69}, \eqref{416} and \eqref{614}hold and
\begin{equation}
    \lim_{n\rightarrow\infty}\sum_{k}\left|\tilde{a}_{nk}-\tilde{a}_{k-1} \right|~~\textrm{exists},
\end{equation}
\begin{equation}\label{417}
    \lim_{n}a_{n}~~\textrm{exist}.
\end{equation}
\end{cor}
\begin{cor}\label{t49}$(i)$ $A=(a_{nk})\in (cs^{\lambda}:c_0)$ if and only if \eqref{48} and \eqref{68} hold and

\begin{equation}
   \lim_{n}\tilde{a}_{nk}=0.
\end{equation}
$(ii)$ $A=(a_{nk})\in (cs_0^{\lambda}:c_0)$ if and only if \eqref{48} and  \eqref{68}  hold and
\begin{equation}
    \lim_{n}(\tilde{a}_{nk}-\tilde{a}_{n,k+1})=0.
\end{equation}
$(iii)$ $A=(a_{nk})\in (bs^{\lambda}:c_0)$ if and only if  \eqref{69} and \eqref{416} hold and
\begin{equation}
    \lim_{n}\sum_{k}|\tilde{a}_{nk}-\tilde{a}_{n,k+1}|=0,
\end{equation}
\begin{equation}
   \lim_{n}a_{n}=0.
\end{equation}
\end{cor}
\begin{cor}\label{t410}(i) $A=(a_{nk})\in (cs^{\lambda}:\ell_{1})$ if and only if  \eqref{48} holds and
\begin{equation}\label{624}
\sum_{k}|\tilde{a}_{nk}-\tilde{a}_{n,k+1}|<\infty,
\end{equation}
\begin{equation}\label{421}
\sup_{N,K\in\mathcal{F}}\left|\sum_{n\in N}\sum_{k\in K}(\tilde{a}_{nk}-\tilde{a}_{n,k-1})\right|<\infty.
\end{equation}
$(ii)$ $A=(a_{nk})\in (cs_0^{\lambda}:\ell_{1})$ if and only if \eqref{48} and \eqref{624} hold and
\begin{equation}\label{422}
\sup_{N,K\in\mathcal{F}}\left|\sum_{n\in N}\sum_{k\in K}(\tilde{a}_{nk}-\tilde{a}_{n,k+1})\right|<\infty.
\end{equation}
$(iii)$ $A=(a_{nk})\in (bs^{\lambda}:\ell_{1})$ if and only if   \eqref{69}, \eqref{416}, \eqref{624}, and \eqref{422} hold and
\begin{equation}
   \sum_{n}|a_{n}|<\infty.
\end{equation}
\end{cor}
\begin{cor}\label{t411}$(i)$ $A=(a_{nk})\in (cs^{\lambda}:\ell_{p})$ if and only if   \eqref{48} and \eqref{624} hold and
\begin{equation}\label{625}
    \sup_{K\in\mathcal{F}}\sum_{n}\left| \sum_{k\in K}(\tilde{a}_{nk}-\tilde{a}_{n,k-1})\right|^{p}<\infty.
\end{equation}
$(ii)$ $A=(a_{nk})\in (cs_0^{\lambda}:\ell_{p})$ if and only if \eqref{48} and \eqref{624} hold and
\begin{equation}\label{627}
    \sup_{K\in\mathcal{F}}\sum_{n}\left|\sum_{k\in K}(\tilde{a}_{nk}-\tilde{a}_{n,k+1})\right|^{p}<\infty.
\end{equation}
$(iii)$ $A=(a_{nk})\in (bs^{\lambda}:\ell_{p})$ if and only if \eqref{69}, \eqref{416},  \eqref{624} and \eqref{627} hold and
\begin{equation}
   \sum_{n}|a_{n}|^{p}<\infty.
\end{equation}
\end{cor}

Since Corollary \ref{t48}, Corollary \ref{t49}, Corollary \ref{t410} and Corollay \ref{t411} can be  proved similarly with  Theorem \ref{t47}, we omit their proofs.

\section{Conclusion}
In the literature, the approach of constructing a new sequence space by means of the matrix domain of a particular limitation method has recently been employed by several authors, for example, $[17-35]$.
They introduced the sequence spaces $(\ell_{\infty})_{N_q}$ and $(c_{N_q}$ in \cite{wang}, $(\ell_p)_{C_1}=X_p$  and $(\ell_{\infty})_{C_1}=X_{\infty}$ in \cite{lee}, $\mu_G=Z(u, \nu; \mu)$ in \cite{emalsavas}, $(\ell_{\infty})_{R^t}=r_{\infty}^t,~c_{R^t}=r_{c}^t $ and  $(c_0)_{R^t}=r_{0}^t $ in \cite{emal},
$(\ell_{p})_{R^t}=r_{p}^t$ in \cite{ab1R}, $(c_0)_{E^r}=e_0^r$ and $c_{E^r}=e_c^r$ in \cite{ab1E}, $(\ell_{p})_{E^r}=e_{p}^r$ and $(\ell_{\infty})_{E^r}=e_{\infty}^r$ in \cite{ab1Em}, $(c_0)_{A^r}=a_0^r$ and $c_{A^r}=a_c^r$ in \cite{cafb1}, $[(c_0)(u, p)]_{A^r}=a_0^r(u,p)$ and $[c(u, p)]_{A^r}=a_c^r(u,p)$ in \cite{cafb2}, $e_0^r(\Delta: p)=(c_0(p))_{E^r\Delta}$, $e^r(\Delta: p)=(c(p))_{E^r\Delta}$ and $e_{\infty}^r(\Delta: p)=(\ell_{\infty}(p))_{E^r\Delta}$ in \cite{vatanh}, $(\ell_p)_{A^r}=a_p^r$ and $(\ell_{\infty})_{A^r}=a_{\infty}^r$ in \cite{cafb3}, $(c_0)_{C_1}=\tilde{c_0}$ and $c_{C_1}=\tilde{c}$ in \cite{msfb1}, $(f_0)_{C_1}=\tilde{f_0}$ and $f_{C_1}=\tilde{f}$ in \cite{mskk1}, $(f_0)_{R^t}=\hat{f_0}$ and $f_{R^t}=\hat{f}$ in  \cite{mskk2}, $\{\ell_{\infty}\}_{B(r,s)}=\hat{\ell}_{\infty}$, $\{c_{B(r,s)}=\hat{c}$, $\{c_0\}_{B(r,s)}=\hat{c_0}$ and $\{\ell_p\}_{B(r,s)}=\hat{\ell}_{p}$ in  \cite{kiris}, $f_{B(r,s)}=\hat{f}$, $\{f_0\}_{B(r,s)}=\hat{f_0}$ in \cite{kiris1}, $\nu_{B(r,s,t)}=\nu(B)$ in \cite{cabbar}, and $f_{B(r,s,t)}=f(B)$ in \cite{cabbar2}; where $N_q,~C_1,~R^t,~$and $E^r$ denote N\"{o}rlund, Ces\`{a}ro, Riesz, and Euler means, respectively. $A^r, G, B(r,s)$ and $B(r,s,t)$ are , respectively, defined in \cite{cafb1, emalsavas, kiris, cabbar},$\mu\in\{c_0,~ c,~\ell_p\},~\nu\in\{\ell_{\infty},~c_0,~ c,~\ell_p\}$ and $1\leq p<\infty.$ Also $c_o(u, p)$ and $c(u,p)$ denote sequence spaces genarated from the Maddox's spaces $c_0(p)$ and $c(p)$ by Ba\c sar{\i}r \cite{basarir}.

Quite recently, Mursaleen and Noman have introduced the sequence spaces $\ell_{\infty}^{\lambda}, c^{\lambda}, c_0^{\lambda}$ and $\ell_{p}^{\lambda}$ in the case $1\leq p \leq\infty$ in \cite{mn01, mn04I}, respectively. Although the matrix $\Lambda$ is used for obtaining some new sequence spaces by its domain from the classical sequence spaces, the triangle matrix $\Lambda$ over the sequence spaces $cs,~cs_0$ and $bs$ is not studied. So, working the domain of $\Lambda$ matrix in the spaces $cs,~cs_0$ and $bs$ is meaningful, which is filling up a gap in the existing literature.

Finally, the domain of  difference matrix in the $cs,cs_{0}, bs$ and $cs^{\lambda}, cs_{0}^{\lambda}, bs^{\lambda}$ was not studied. We conclude our work by expressing from now on that the aim of our next papers is to investigate the difference spaces $cs^{\lambda}(\Delta),~ cs_{0}^{\lambda}(\Delta)$ and  $bs^{\lambda}(\Delta).$

\end{document}